\documentclass[11pt]{article}
\usepackage{mathrsfs}
\usepackage{amsfonts}
\usepackage{amsmath}
\usepackage{amssymb}
\usepackage{amsopn}
\usepackage{amsthm}
\usepackage{amstext}
\def\ep{\epsilon}
\def\de{\delta}

\def\lam{\lambda}
\def\la{\langle}
\def\ra{\rangle}
\def\nbb{(\mathbf \beta_1,\cdots,\beta_n)}
\def\naa{(\mathbf \alpha_1,\cdots,\alpha_n)}
\def\naaa{\mathbf\{\alpha_1,\cdots,\alpha_n\}}
\def\nee{(\mathbf\ep_1,\cdots,\ep_n)}

\def\eb{{\mathbf e}}

\def\na{(a_1,\cdots,a_n)}
\def\nep{(\ep_1,\cdots,\ep_n)}
\def\ol{\overline}
\def\inn{i=1,\cdots,n}
\def\mmaa{(\la{\mathbf\alpha_i,\alpha_j}\ra)}

\def\inn{i=1,\cdots,n}
\def\aa{\mathcal{A}}

\newtheorem{lemma}{Lemma}[section]
\newtheorem{theorem}[lemma]{Theorem}
\newtheorem{corollary}[lemma]{Corollary}

\newtheorem{remark}[lemma]{Remark}

\newtheorem{example}[lemma]{Example}

\begin{document}
\date{}
\title{ Least Squares Problems in  Orthornormalization   }
\author{Shanwen Hu\thanks{E-mail: swhu@math.ecnu.edu.cn}\\
  Research Center for Operator Algebras\\
 East China Normal University, Shanghai 200241, P.R. China
}

\maketitle

\begin{abstract} \noindent
For any $n$-tuple $\naa$ of linearly independent vectors in Hilbert
space $H$, we construct a unique orthonormal basis $\nee$ of
$span\{\alpha_1,\cdots,\alpha_n\}$ satisfying:
$$\sum_{i=1}^n\|\ep_i-\alpha_i\|^2\le\sum_{i=1}^n\|\beta_i-\alpha_i\|^2$$
 for all orthonormal basis $\nbb$ of $span\{\alpha_1,\cdots,\alpha_n\}$.
 We study the stability of the orthornormalization and give some
applications and examples.

\vspace{3mm}


\vspace{3mm} \noindent{2010 {\it Mathematics Subject
Classification\/}: 46C05,15A60,65J05}
\end{abstract}

\section{Notations and Introduction}
Throughout this paper,
\begin{enumerate}

\item $\mathbb C$ (or $\mathbb R$) is the complex (or real) number
field.

\item  For any $z\in \mathbb C$, $\ol{z}$ is the complex conjugate of $z$.
$Re(z)$ is the real part of $z$.

\item  $M_{n,\; m}(\mathbb C)=\{(a_{ij})\}$ is the set of  $n\times m$ complex
matrices. $M_n(\mathbb C)=M_{n,\;n}(\mathbb C)$.

\item $\mathbb C^n=M_{n,1}(\mathbb C)$. The identity of $\mathbb
C^n$ denoted by $I_n$.

\item  For any $(a_{ij})\in M_{n,\;m}(\mathbb C)$, $\ol{(a_{ij})}=(\ol{a}_{ij})\in M_{n,\; m}(\mathbb C)$.

\item The
standard basis  of $\mathbb C^n$ is denoted  by ${\mathbf
e_i}=(\de_{1i},\cdots,\de_{in}),\inn$.

\item For any $(a_{ij})\in M_n(\mathbb C)$.  $tr(a_{ij})=\sum_{i=1}^na_{ii}$. $\sigma\big((a_{ij})\big)$
is the set of spectrum of $(a_{ij})$.
$\|(a_{ij})\|=\sup\{\|(a_{ij})x^T\|:x\in \mathbb C^n,\|x\|=1\}$.

\item  $H$ is a  complex  Hilbert space with the inner product
$\langle\cdot,\cdot\rangle$.

\item  $H^n=\{\naa:\alpha_i\in H,i=1,\cdots,n\}$ is a Hilbert space,
the inner product defined by:
 $\la\naa,\nbb\ra=\sum_{i=1}^n\la\alpha_i,\beta_i\ra$.

\item For any $(a_{ij})\in M_{n,\; m}(\mathbb C),(a_{ij})^T=(a_{ji})\in M_{m,\;n}(\mathbb C)$ and for any $\naa\in H^n$,
$$\naa^T=\begin{pmatrix}
 \alpha_1\\\vdots\\\alpha_n
\end{pmatrix}.$$

\end{enumerate}

 If $\{\alpha_1,\cdots,\alpha_n\}\subset H^n$ is subset of linearly independent
 vectors in $H$.
The orthonormalization of $\{\alpha_1,\cdots,\alpha_n\}$ is to find
a solution $\{\beta_1,\cdots,\beta_n\}$ in
$span\{\alpha_1,\cdots,\alpha_n\}$ to the system:  for any $i,j$,

$$\big\la \beta_i,\beta_j\big\ra=\de_{ij}=\left\{\begin{array}{c}
1,i=j\\0,i\not=j\end{array}\right.,\quad i,j=1,\cdots,n.$$

The orthonormalization can be carried out in infinitely ways. The
most simple and practical tool is the Gram-Schmidt process. It is a
recursive process and are used widely in various fields.  However in
general, the Gram-Schmidt process can not provide a simple
construction formula for the solution and can not give a   method to
calculate the sum of squares
$$
\|(\beta_1,\cdots,\beta_n)-\naa\|^2=
\sum_{i=1}^n\|\beta_i-\alpha_i\|^2.
$$
Gram-Schmidt process is also unstable due to the repeated various
operations. These restrict its applications, especially in the
abstract or  theoretical analysis.

In numerical linear algebra,  Householder method is also used in the
orthornormalization.   The Gram-Schmidt process produces the jth
orthogonalized vector after the jth iteration, while Householder
method produces all vectors only at the end. And theoretically
Household method take twice operations as Gram-Schmidt process, but
it uses orthogonal transformation at each iteration, so it is
stable. The Household method is restricted in numerical linear
algebra only.

The  Gram-Schmidt process and Householder method  can be find in
Linear Algebra and Matrix Analysis text books, for example, see [1]
or [2].

In Section 2, we provide a  simple and uniform formula $K\naa$ (see
(2.4)) for any Hilbert space, on complex field or  real field, with
finite or infinite dimensional, in numerical form or not in
numerical form, no iteration,
  to   construct an
orthonormal basis of  $span\{\alpha_1,\cdots,\alpha_n\}$, only
according the direct information
$\{\la\alpha_i,\alpha_j\ra:i,j=1,\cdots,n\}$, satisfying:
$$
\sum_{i=1}^n\|\ep_i-\alpha_i\|^2=n+\sum^n_{i=1}\|\alpha_i\|^2-2tr
\big((\la \alpha_i,\alpha_j\ra)^{1/2}\big).
$$
 Moreover, we  show $n+\sum_{i=1}^n\|\alpha_i\|^2-2tr\big((\la
\alpha_i,\alpha_j\ra)^{1/2}\big)$ is the lower bound of all  sum
squares of  orthonormal basis of $span\naaa$, and $\nee$ is the
unique one minimized the sum of squares.

In Section 3, we show  our construction is stable in the sense: for
any given $\ep>0$ and any linearly independent $n$-tuple $\naa$ in
$H^n$, there exists $\de$, dependent only on $\naa$ and $\ep$, such
that for any $\nbb\in H^n$, $\max\{\|\alpha_i-\beta_i\|\}<\de$
implies $\|K\naa-K\nbb\|<\ep$.

In section 4, as an application, we establish a formula for the
distance of between any $\gamma\in H$ and  $span\naaa$, generalize
the one in the case $\naaa$ is an orthonormal basis. For any
$\ep$-mutually orthonormal basis $\naa$,(see (4.4) for the
definition), we show there is  an orthonormal basis $\nee$, the
distance between them in $H^n$  not exceed $\sqrt{2(n-1)\ep}$.

\section{ Least Square in  Orthornormalization}

(1) For any $\naa\in H^n$, $\na\in \mathbb C^n$, if
$\eta=\sum_{i=1}^na_i\alpha_i$,  we write
$$
\eta=\na
\naa^T=\naa\begin{pmatrix}\alpha_1\\\vdots\\\alpha_n\end{pmatrix}.
$$
(2) If for any
 $i$,  $\beta_i=(a_{i1},\cdots,a_{in})\naa^T,i=1,\cdots,n$, then we
write
$$
\begin{pmatrix}
 \beta_1\\\vdots\\\beta_n
\end{pmatrix}
=\left(\begin{array}{ccc}
 a_{11}&\cdots&a_{1n}\\
 \vdots&&\vdots\\
 a_{n1}&\cdots&a_{nn}
\end{array}\right)\begin{pmatrix}
 \alpha_1\\\vdots\\\alpha_n
\end{pmatrix}.\eqno(2.1)
$$
That is
$$\beta_i={\mathbf e}_i(a_{ij})\naa^T,i=1,\cdots,n.$$
It is easy to check:  if $\eta=\sum_i^na_i\alpha_i$ and
$\zeta=\sum_{j=1}^mb_j\beta_j$, then
$$\la\eta,\zeta\ra
=\sum_{j=1}^m\sum_{i=1}^na_i\ol{b}_j\la\alpha_j,\beta_i\ra
=(\ol{b}_1,\cdots,\ol{b}_m)(\la\alpha_j,\beta_i\ra)\na^T.\eqno(2.2)$$

\begin{theorem} Suppose $\naa$ is $n$-tuple of linearly  independent
vectors in a complex  Hilbert space, then

(1) For any orthonormal base $(\beta_1,\cdots,\beta_n)$ of
$span\{\alpha_1,\cdots,\alpha_n\}$,  there exists an invertible
$(a_{ij})\in M_n(\mathbb C)$ such that (2.1) holds and satisfying,

$$\sum_{i=1}^n\|\beta_i-\alpha_i\|^2=
n+\sum_{i=1}^n\|\alpha_i\|^2-2Re\Big(tr\big((\la\alpha_j,\alpha_i\ra)(a_{ji})\big)\Big).
\eqno (2.3)
$$

 (2) $(\la\alpha_j,\alpha_i\ra)$ is positive definite and if   $K\naa=\nee$   defined by
$$\begin{pmatrix}
 \ep_1\\\vdots\\\ep_n
\end{pmatrix}
=\left(\begin{array}{ccc}
 \la\alpha_1,\alpha_1\ra&\cdots&\la\alpha_1,\alpha_n\ra\\
 \vdots&&\vdots\\
 \la\alpha_n,\alpha_1\ra&\cdots&\la\alpha_n,\alpha_n\ra
\end{array}\right)^{-1/2}
\begin{pmatrix}
 \alpha_1\\\vdots\\\alpha_n
\end{pmatrix},\eqno(2.4)
$$
then

(a)  $\nep$ is a orthonormal basis  of
$span\{\alpha_1,\cdots,\alpha_n\}$.

(b)

$$\sum_{i=1}^n\|\ep_i-\alpha_i\|^2=n+\sum_{i=1}^n\|\alpha_i\|^2-2tr\Big(\big(\la\alpha_j,\alpha_i\ra\big)^{1/2}\Big).\eqno(2.5)$$

(c) For any orthonormal base $(\beta_1,\cdots,\beta_n)$ of
$span\{\alpha_1,\cdots,\alpha_n\}$ ,
$$\sum_{i=1}^n\|\ep_i-\alpha_i\|^2\le\sum_{i=1}^n\|\beta_i-\alpha_i\|^2, $$
and the equality holds if and only if
$(\beta_1,\cdots,\beta_n)=\nep$.
\end{theorem}

\begin{proof}
(1) There is no problem for the existence and the invertibility of
$(a_{ij})$. Noticing $\beta_i=\eb_i(a_{ij})\naa^T$ and
$\alpha_i=\eb_i\naa^T$, applying (2.2), we have
\begin{eqnarray*}
&&\sum_{i=1}^n\|\beta_i-\alpha_i\|^2\\
&=&n+\sum_{i=1}^n\|\alpha_i\|^2-\sum_{i=1}^n\Big(\la\beta_i,\alpha_i\ra+\la\alpha_i,\beta_i\ra\Big)\\
&=&n+\sum_{i=1}^n\|\alpha_i\|^2-2Re\Big(\sum_{i=1}^n\big\la\eb_i(a_{ij})\naa^T,\eb_i\naa^T\big\ra\Big)\\
&=&n+\sum_{i=1}^n\|\alpha_i\|^2-2Re\Big(\sum_{i=1}^n\eb_i\big((\la\alpha_j,\alpha_i\ra)(a_{ij})^T\big)\eb^T_i\Big)\\
&=&n+\sum_{i=1}^n\|\alpha_i\|^2-2Re\Big(tr\big((\la\alpha_j,\alpha_i\ra)(a_{ji})\big)\Big).
\end{eqnarray*}

(2) (a) For any $(a_1,\cdots,a_n)\in\mathbb C^n$, not all zero, by
(2.1) 
$$(\ol{a}_1,\cdots,\ol{a}_n)(\la\alpha_j,\alpha_i\ra)(a_1,\cdots,a_n)^T=\Big\la\sum_{i=1}^na_i\alpha_i,\sum_{i=1}^na_i\alpha_i\Big\ra>0,$$
so $(\alpha_j,\alpha_i\ra)$ is positive definite.  Noticing
$\ol{(\la\alpha_i,\alpha_j\ra)^{-1/2}}=(\ol{\la\alpha_i,\alpha_j}\ra)^{-1/2},$
and applying  (2.1),  we have
$$
\la\ep_i,\ep_j\ra=\eb_j(\ol{\la\alpha_i,\alpha_j\ra})^{-1/2}(\la\alpha_j,\alpha_i\ra)\big((\la\alpha_i,\alpha_j\ra)^T\big)^{-1/2}\eb_i^T
=\eb_j\eb^T_i=\de_{ji},
$$
for any $i,j$.

(b) (2.5) follows from (2.3) and
$(a_{ji})=(\la\alpha_j,\alpha_i\ra)^{-1/2}$.

(c) There exists a unitary $U=(u_{ij})\in M_n(\mathbb C)$ such that
$$ \begin{pmatrix}
 \beta_1\\\vdots\\\beta_n
\end{pmatrix}
=\left(\begin{array}{ccc}
 u_{11}&\cdots&u_{1n}\\
 \vdots&&\vdots\\
 u_{n1}&\cdots&u_{nn}
\end{array}\right)\begin{pmatrix}
 \ep_1\\\vdots\\\ep_n
\end{pmatrix}.
$$
Then
$$ \begin{pmatrix}
 \beta_1\\\vdots\\\beta_n
\end{pmatrix}
=\left(\begin{array}{ccc}
 u_{11}&\cdots&u_{1n}\\
 \vdots&&\vdots\\
 u_{n1}&\cdots&u_{nn}
\end{array}\right)
\left(\begin{array}{ccc}
 \la\alpha_1,\alpha_1\ra&\cdots&\la\alpha_1,\alpha_n\ra\\
 \vdots&&\vdots\\
 \la\alpha_n,\alpha_1\ra&\cdots&\la\alpha_n,\alpha_n\ra
\end{array}\right)^{-1/2}\begin{pmatrix}
 \alpha_1\\\vdots\\\alpha_n
\end{pmatrix},
$$
so we have
$$(a_{ij})=(u_{ij})\big(\la\alpha_i,\alpha_j\ra\big)^{-1/2}\quad\quad or
\quad\quad (a_{ji})=(\la\alpha_j,\alpha_i\ra)^{-1/2}(u_{ji}).$$
Therefore
$$
Re\Big(tr\big((\la\alpha_j,\alpha_i\ra)(a_{ji})\big)\Big)
=Re\Big(tr\big((\la\alpha_j,\alpha_i\ra)^{1/2}(u_{ji})\big)\Big).
$$

To the end of the proof, it is enough to show the statement: for any
positive $T\in M_n(\mathbb C)$ and any unitary $U$, $Re\big( tr
(TU)\big)\le tr(T)$, and the equation holds if and only if $U=I_n$.
Since $T$ may be assumed diagonal, it is easy to prove the
statement. We omit the details.
\end{proof}

\begin{remark}
 If $H$ is a real complete inner product space, the conclusions in
the Theorem 2.1 are  still  true. In fact, the inverse and squares
root of a symmetric matrix are all real. Noticing
$\la\alpha_i,\alpha_j\ra=\la\alpha_j,\alpha_i\ra$,  the proof has
nothing to be changed  except
 the unitary matrix replaced by a orthogonal matrix.
 \end{remark}

\begin{example}
Suppose $H=\mathbb R^m$,
$\alpha_i=(a_{i1},\cdots,a_{im}),i=1,\cdots,n$. If $\naa$ is
linearly independent, then $n\le m$ and

$$
A=\left(\begin{array}{ccc}
 a_{11}&\cdots&a_{1m}\\\
 \vdots&&\vdots\\
 a_{n1}&\cdots&a_{nm}
\end{array}\right)
 $$
has rank $n$, so $AA^T$ is invertible.  If
$(AA^T)^{-1/2}A=(h_{ij})$, then $K\naa=(\ep_1,\cdots,\ep_n)$ defined
by: $\ep_i=(h_{i1},\cdots,h_{im}),\inn$.

For any other orthonormal basis $(\beta_1,\cdots,\beta_n)$ of
$span\{\alpha_1,\cdots,\alpha_n\}$, if
$$\beta_i=(b_{i1},\cdots,b_{im}),i=1,\cdots,n,$$ then
\begin{eqnarray*}
\sum_{i=1}^n\|\ep_i-\alpha_i\|^2=\sum_{j=1}^m\sum_{i=1}^n|a_{ij}-h_{ij}|^2
&<&\sum_{j=1}^m\sum_{i=1}^n|b_{ij}-h_{ij}|^2=\sum_{i=1}^n\|\beta_i-\alpha_i\|^2.
\end{eqnarray*}
\end{example}
\begin{example}
Let $H=L^2([0,1])$ be the space of real function which square is
Lebesgue integrable.  For any $n$,  $\{1,x,\cdots,x^n\}$ is linearly
independent. By the definition of (2.4),

$$
\begin{pmatrix}
 \ep_1\\
 \ep_2\\
 \ep_3\\
 \vdots\\
 \ep_n
\end{pmatrix}
=\left(\begin{array}{ccccc}
 1&\frac{1}{2}&\frac{1}{3}&\cdots&\frac{1}{n-1}\\
 \frac{1}{2}&\frac{1}{3}&\cdots&\cdots&\cdots\\
 \frac{1}{3}&\cdots&\cdots&\cdots&\cdots\\
 \vdots&&&&\\
 \frac{1}{n-1}&\cdots&\cdots&\cdots&\frac{1}{2n-1}
\end{array}\right)^{-1/2}
\begin{pmatrix}
 1\\
 x\\
 x^2\\
 \vdots\\
 x^{n-1}
\end{pmatrix}
$$
If the transformation matrix denoted by $A_n$, then
$$
\sum_{i=1}^n\|\alpha_i-\ep_i\|^2=n+1+\frac{1}{3}+\frac
15+\cdots+\frac{1}{2n-1}-2tr(A^{-1}_n).$$

For  $n=4$, with the help of MATLAB,  we immediately get:

$$
\left\{\begin{array}{ccccccccc}
\ep_1&=&\ \ \;1.8145&-&\ \,2.8273x&+&\ \,2.0557x^2&-&\ \,0.6986x^3\\
\ep_2&=&\,-2.8273&+&18.1940x&-&26.7823x^2&+&11.9872x^3\\
\ep_3&=&\ \ \;2.0557&-&26.7823x &+&64.5308x^2  &-&39.9282x^3\\
\ep_4&=& \ \!-0.6986 &+&11.9872x&-&39.9282 x^2&+&32.5816x^3
\end{array}\right.
$$
and
 $$
\sum_{i=1}^4\|x^{i-1}-\ep_i\|^2=4+1+\frac{1}{3}+\frac
15+\frac{1}{7}-2tr(A_4^{-1})= 2.2201,$$ the accuracy is controlled
by MATLAB.
\end{example}

\section{Perturbations}

In this section, we will discuss the perturbation problem and show
our construction  is stable in some sense.

Note
 although $(\la\alpha_i,\alpha_j\ra)$ is non singular,  it may be
 nearly singular when $\|\alpha_i\|$ is too small. We may modify our
 work, replacing
 $\naa$ by
 $(\frac{\alpha_1}{\|\alpha_1\|},\cdots,\frac{\alpha_n}{\|\alpha_n\|})$
 in this section.

To show the main result of the section, we need the following
Lemmas. Lemma 3.1 is well known.
\begin{lemma}
 If $A$ is an invertible element of unital $C^*$-algebra $\aa$,  for
 any $B\in\aa$ with $\|B-A\|<\|A^{-1}\|^{-1}$, then $B$ is
 invertible satisfying
 $$
 \|A^{-1}-B^{-1}\|\le\frac{\|A^{-1}\|^2\|A-B\|}{1-\|A^{-1}\|\|A-B\|}\eqno(3.1)
 $$
 and
 $$
 \|B^{-1}\|\le \frac{\|A^{-1}\|}{1-\|A^{-1}\|\|A-B\|}.\eqno(3.2)
 $$
 \end{lemma}
 \begin{lemma} If $A,B$ are positive and invertible on a Hilbert space $H$,
 then
 $$
 \|A^{-1/2}-B^{-1/2}\|\le\|A^{1/2}\|\|A^{-1}-B^{-1}\|.\eqno(3.3)
 $$
 \end{lemma}
 \begin{proof}
 It is enough to show
 $$
 \|A-B\|\le\|A^{-1}\|\|A^2-B^2\|.\eqno(3.4)
 $$
There exists $\lambda\in \sigma(A-B)$ with $|\lambda|=\|A-B\|$ and $
\{x_n\}\subset H,\|x_n\|=1,i=1,2,\cdots$ such that
$$\lim_{n\to\infty}\big\|(A-B)x_n-\lambda x_n\|=0.$$

Since $A^2-B^2=A(A-B)+(A-B)A-(A-B)(A-B)$, we have
\begin{eqnarray*}
&&\|A^2-B^2\|\\
&\ge&\liminf_{n\to\infty}\big|\big\la(A^2-B^2))x_n,x_n\big\ra\big|\\
&=&\liminf_{n\to\infty}\big|\big\la A(A-B)x_n,x_n\big\ra+\big\la(A-B)Ax_n,x_n\big\ra-\big\la(A-B)(A-B)x_n,x_n\big\ra\big|\\
&=&\liminf_{n\to\infty}\big|\big\la \lambda Ax_n,x_n\big\ra+\big\la\lambda Ax_n,x_n\big\ra-\big\la\lambda(A-B)x_n,x_n\big\ra\big|\\
&=&\liminf_{n\to\infty}\big||\lambda||\big\la Ax_n,x_n\big\ra+\big\la Bx_n,x_n\big\ra|\big|\\
&\ge&\|A-B\|\|A^{-1}\|^{-1},
\end{eqnarray*}
thus we obtain (3.4) and (3.3).
\end{proof}
The special conditions of the following Lemma 3.3  are prepared for
the proof of Theorem 3.4.
\begin{lemma}
Suppose $A,B,C\in M_n(\mathbb C)$, $B$ is positive,  $A$ is positive
and invertible with $\|A\|\le n,1\le\|A^{-1}\|$.  If there exists
$\eta$ with $0\le\eta\le 2\|A^{-1}\|^{-1})$, $\|A-B\|<\eta$ and
$\|A-C\|<\eta,$
 then $B$ is invertible and
 $$\|A^{-1/2}CB^{-1/2}-I_n\|\le 2n(n+1)\|A^{-1}\|^2\eta.\eqno(3.5)$$
 \end{lemma}
 \begin{proof}
 It follows Lemma 3.1, $B$ is invertible and
 $$
 \|A^{-1}-B^{-1}\|<2\|A^{-1}\|^2\|A-B\|\le2\|A^{-1}\|^2\eta,\eqno(3.6)
 $$
 $$
 \|B^{-1}\|\le 2\|A^{-1}\|.\eqno(3.7)
 $$
Condition $\|A^{-1}\|\ge 1$ implies $\|A^{-1}\|\ge \|A^{-1/2}\|$ and
 $$\|B^{-1/2}\|=\|B^{-1}\|^{1/2}\le(2\|A^{-1}\|)^{1/2}\le
 2\|A^{-1}\|.\eqno(3.8)
 $$
 By (3.3) and (3.6),
$$
\|A^{-1/2}-B^{-1/2}\|\le
\|A^{1/2}\|\|A^{-1}-B^{-1}\|\le2\|A^{1/2}\|\|A^{-1}\|^2\eta.
$$
At last, applying (3.7) and (3.8), we get
\begin{eqnarray*}
&&\|A^{-1/2}CB^{-1/2}-I_n\|\\
&\le&\|A^{-1/2}CB^{-1/2}-A^{-1/2}AB^{-1/2}\|+\|A^{1/2}B^{-1/2}-A^{1/2}A^{-1/2}\|\\
&\le&\|A^{-1/2}\|\|C-A\|\|B^{-1/2}\|+\|A^{1/2}\|\|B^{-1/2}-A^{-1/2}\|\\
&\le&2\|A^{-1}\|^2\eta+2\|A\|\|A^{-1}\|^2\eta\\
&=&2(n+1)\|A^{-1}\|^2\eta.
\end{eqnarray*}
\end{proof}
\begin{theorem}
Suppose $\naa$ is an $n$-tuple of linearly independent units of
$H^n$, for any given $\ep>0$, let
 $$\de=\big(8n^2(n+1)\|(\la\alpha_j,\alpha_i\ra)^{-1}\|^2\big)^{-1}\ep,$$ then for any
 $\nbb$ of units in $H^n$, if
$$\max\{\|\alpha_i-\beta_i\|:i,j=1,\cdots,n\}<\de,\eqno(3.9)$$
then  $\nbb$  is linearly independent and satisfying
$$\|K\naa-K\nbb\|^2<\ep.\eqno(3.10)
$$
\end{theorem}
\begin{proof}
Let
$A=(\la\alpha_j,\alpha_i\ra),B=(\la\beta_j,\alpha_i\ra),C=(\la\beta_j,\beta_i\ra)$.
Since $tr(A)=n$, so $\lambda_{min}\le 1$, where
$\lambda_{min}=\min\big(\sigma(A)\big),$ this implies
$\|A^{-1}\|=\lambda^{-1}_{min}\ge 1$. $\|A\|\le n $ follows from for
any $i,j$, $|\la\alpha_i,\alpha_j\ra|\le 1$.
 The condition (3.9) guarantee for any $i,j$, $$|\la\alpha_j,
\alpha_i\ra-\la\beta_j,\beta_i\ra\|<2\de,\quad|\la\alpha_j,\alpha_i\ra-\la\alpha_j,\beta_i\ra|<\de.$$
Consequently,
$$\|A-B\|<2n\de<\big(2\|A^{-1}\|^2\big)^{-1},\quad\quad\|A-C\|<2n\de.$$
Let $\eta=2n\de$, then all conditions in Lemma 3.3 are all
satisfied, so we have
$$\|A^{-1/1}CB^{-1/2}-I_n\|\le 2(n+1)\|A^{-1}\|^2\eta=4n(n+1)\|A^{-1}\|^2\de.$$
The  linearly independence of $\nbb$ follows from the invertibility
of $B$.

Now suppose  $K\naa=\nee,K\nbb=(\tau_1,\cdots,\tau_n)$, then by the
definition (2.4) and formula (2.2),
\begin{eqnarray*}
&&\sum_{i=1}^n\|\ep_i-\tau_i\|^2=2n-\sum_{i=1}^n\big(\la\ep_i,\tau_i\ra+\la\tau_i,\ep_i\ra\big)\\
&=&2n-\sum_{i=1}^n2Re\Big(\eb_i(\la\alpha_j,\alpha_i\ra)^{-1/2}(\la\beta_j,\alpha_i\ra)(\la\beta_j,\beta_i\ra)^{-1/2}\eb^T_i\Big)\\
&=&2Re\Big(tr(I-A^{-1/2}CB^{-1/2})\Big)\\
&\le&2\|I_n-A^{-1/2}CB^{-1/2}\|tr(I_n)\le
8n^2(n+1)\|A^{-1}\|^2\de=\ep,
\end{eqnarray*}
thus we complete the proof.
\end{proof}
\section{Applications}
In this section, we will give some applications of our construction
(2.4) to the theoretical  analysis.

For a fixed linearly independent vectors $\naaa$, for any $\gamma\in
H$, we define
$$D(\gamma)=dist(\gamma,span\naaa)=\inf\{\|\gamma-\beta\|:\beta\in span\naaa\}.$$
If $\naaa$ is mutual orthogonal units, then
$$D(\gamma)=\sqrt{\|\gamma\|^2-\sum_{i=1}^n|\la\gamma,\alpha_i\ra|^2}.\eqno(4.1)$$
The following theorem will show (4.1) is just (4.2) in the special
case.

For simple, in this section, for any $\naa \in H^n$, we define
$$\naa\circ\naa=\big(\la\alpha_i,\alpha_j\ra\big)\in M_n(\mathbb C).$$

\begin{theorem}
 Suppose $\naa$ is $n$-tuple linearly  independent
vectors in a complex  Hilbert space, then
$$
D(\gamma)
=\sqrt{\frac{det\big((\gamma,\alpha_1\cdots,\alpha_n)\circ(\gamma,\alpha_1,\cdots,\alpha_n)\big)}
{det\big((\alpha_1,\cdots,\alpha_n)\circ(\alpha_1,\cdots,\alpha_n)\big)}}.\eqno(4.2)
$$
\end{theorem}
\begin{proof} Suppose $\nee=K\naa$ defined by (2.4).
Let $\Delta=det\mmaa$, applying  formula (2.2) and
$\sum_{i=1}^n\eb^T\eb_i=I_n$, we have
\begin{eqnarray*}
&&\sum_{i=1}^n|\la\gamma,\ep_i\ra|^2\\
&=&\sum_{i=1}^n|\eb_i(\la\alpha_j,\alpha_i\ra)^{-1/2}(\la\gamma,\alpha_1\ra,\cdots,\la\gamma,\alpha_n\ra)^T|^2\\
&=&\sum_{i=1}^n(\la\alpha_1,\gamma\ra,\cdots,\la\alpha_n,\gamma\ra)(\la\alpha_j,\alpha_i\ra)^{-1/2}\eb_i^T\eb_i(\la\alpha_j,\alpha_i\ra)^{-1/2}(\la\gamma,\alpha_1\ra,\cdots,\la\gamma,\alpha_n\ra)^T\\
&=&\big(\la\alpha_1,\gamma\ra,\cdots,\la\alpha_n,\gamma\ra)(\la\alpha_j,\alpha_i\ra)^{-1}(\la\gamma,\alpha_1\ra\cdots\la\gamma,\alpha_n\ra)^T.\quad\quad
\quad\quad\quad\quad\quad(*)
\end{eqnarray*}
Let $A_{ij}$ be the $(i,j)$ cofactors in $(\la\alpha_j,\alpha_i\ra)$
(not in $\mmaa$!), then

\begin{eqnarray*}
(*)&=&\frac{1}{\Delta}\big(\la\alpha_1,\gamma\ra,\cdots,\la\alpha_n,\gamma\ra)
\left(\begin{array}{ccc}
 A_{11}&\cdots&A_{n1}\\
 \vdots&&\vdots\\
 A_{1n}&\cdots&A_{nn}
\end{array}\right)
\begin{pmatrix}
 \la\gamma,\alpha_1\ra\\\vdots\\\la\gamma,\alpha_n\ra
\end{pmatrix}\\
&=&\frac{1}{\Delta}\Big(\sum_{j=1}^n\sum_{i=1}^nA_{ji}\la\gamma,\alpha_i\ra\la\alpha_j,\gamma\ra\Big)\\
&=&\frac{-1}{\Delta}det\left(\begin{array}{cccc}
0&\la\gamma,\alpha_1\ra&\cdots&\la\gamma\,\alpha_1\ra\\
 \la\alpha_1,\gamma\ra&\la\alpha_1,\alpha_1\ra&\cdots&\la\alpha_1,\alpha_n\ra\\
 \vdots&&&\vdots\\
 \la\alpha_n,\gamma\ra&\la
 \alpha_n,\alpha_1\ra&\cdots&\la\alpha_n,\alpha_n\ra
\end{array}\right)\\
&=&\|\gamma\|^2-\frac{{det((\gamma,\alpha_1,\cdots,\alpha_n)\circ(\gamma,\alpha_1,\cdots,\alpha_n))}}{\Delta}.
\end{eqnarray*}
Since
$$dist (\gamma,span\{\alpha_1,\cdots,\alpha_n\})
 =\sqrt{\|\gamma\|^2-\sum_{i=1}^n|\la\gamma,\ep_i\ra|^2},$$  we obtain
 (4.2).
 \end{proof}
\begin{corollary}
Suppose $\{\alpha_1,\alpha_2,\cdots\}$ is a sequence of independent
vectors, then for any $\gamma\in H$, the distance between $\gamma$
and  the closure of $span\{\alpha_1,\alpha_2,\cdots\}$ is:
$$
\lim_{n\to\infty}\sqrt{\frac{det((\gamma,\alpha_1\cdots,\alpha_n)\circ(\gamma,\alpha_1,\cdots,\alpha_n))}
{det((\alpha_1,\cdots,\alpha_n)\circ(\alpha_1,\cdots,\alpha_n))}}.
$$

\end{corollary}

\begin{lemma}
 Suppose $E=\{\lambda_i,i=1,\cdots,n\}$ is a finite positive
numbers set with $\sum_{i=1}^n\lambda_i=n$ and
$\max\{|1-\lambda_i|:i=1,\cdots,n\}=\ep\le\frac{1}{2(n-1)},$  then
$$\Big|\sum_{i=1}^n\lambda_i^{1/2}- n\Big|\le\ep.\eqno(4.3)$$
\end{lemma}
\begin{proof}
 We may assume
$\lambda_1=\min\big(E\big)$ and  $\lambda_n=\max\big(E\big).$

If $\lam_1=1$ or $\lam_n=1$, then for all $i,\lambda_i=1$, there is
nothing to do. If $\lam_n=1+\eta$, then $0<\eta\le\ep$. Since the
function
$$\sum_{i=1}^{n-2}\lambda_i^{1/2}+\big(n-1-\eta-\sum_{i=1}^{n-2}\lambda_i\big)^{1/2}$$
on
$$[1-\frac{1}{2(n-1)},1+\frac{1}{2(n-1)}]\times\cdots\times[1-\frac{1}{2(n-1)},1+\frac{1}{2(n-1)}]\subset \mathbb R^{n-2}$$
obtains its maximum value only when
$\lambda_i=1-\frac{\eta}{n-1},i=1,\cdots,n-2$,
 consequently,
$$
\sum_{i=1}^n\lambda_i^{1/2}
 \le(1+\eta)^{1/2}+\sum_{i=1}^{n-1}\Big(1-\frac{\eta}{n-1}\Big)^{1/2}\le
 n+\eta.
$$
If $\lambda_1=1-\zeta$,  similar argument show
$$\sum_{i=1}^n\lam_i^{1/2}\ge n-\zeta.$$
Then (4.3) follows the assumption $\eta\le\ep$ and $\zeta\le\ep$.

\end{proof}

If  $\naaa$ satisfies condition,

$$
\ep=\max\{|\la\alpha_i,\alpha_j\ra|:i,j=1,\cdots,n\},\eqno(4.4)
$$
we will say $\naaa$ is  $\ep$-mutually orthogonal.

 In [3], Hu and Xue proved, if $\naaa$ is  $\ep$-mutually
 orthogonal, then there are mutually orthogonal
 $\{\beta_1,\cdots,\beta_n\}$ with $\|\alpha_i-\beta_i\|<6(n-1)\ep,i=1,\cdots,n$.
 Now, we have the following:
\begin{theorem}
Suppose $n$-tuple units $\naa$ in $H$   is $\ep$-mutually orthogonal
with $\ep<\frac{1}{2(n-1)}$, then there is an orthonormal basis
$\nee$ of $span\{\alpha_1,\cdots,\alpha_n\}$ such that
$$\sum_{i=1}^n\|\ep_i-\alpha_i\|^2<2(n-1)\ep.\eqno(4.5)
$$
\end{theorem}
\begin{proof} Let $T=\mmaa$ and assume
$\sigma(T)=\{\lambda_1,\cdots,\lambda_n\}$. Since
$\la\alpha_i,\alpha_i\ra=1$ and
$$\sum_{j\not=i}|\la\alpha_i,\alpha_j\ra|<(n-1)\ep,$$ by Gersgorin
Theorem ([2]Theorem 6.1.1), for all $i,|\lambda_i-1|\le(n-1)\ep$.
Meanwhile,
$$\sum_{i=1}^n\lambda_i=tr(T)=\sum_{i=1}^n\la\alpha_i,\alpha_i\ra=n,$$
then by the Lemma 4.1,
$$\Big|tr(T^{1/2})-n\Big|=\Big|\sum_{i=1}^n\lam_i^{1/2}-n\Big|\le(n-1)\ep.$$
Let $\nee=K\naa$ defined by (2.4), applying (2.5)
$$\sum_{i=1}^n\|\ep_i-\alpha_i\|^2=2\big|tr(T^{1/2})-n\big|=2\Big|\sum_{i=1}^n\lambda_i^{1/2}-n\Big|\le2(n-1)\ep.$$

\end{proof}

{\bf Acknowledgements} \quad The author  would like to thank
Professor Huaxin Lin for his helpful comments and  suggestions from
their  multiple discussions.

\enddocument

\end{document}